\documentclass[11pt,twoside,a4paper]{amsart}
\usepackage{amssymb}
\usepackage{amsmath,amscd}
\usepackage[initials,nobysame]{amsrefs}
\usepackage{hyphenat}
\usepackage[utf8]{inputenc}
\usepackage[all,cmtip]{xy}
\usepackage{mathtools}

\usepackage{geometry}
\def\la{\langle}
\def\ra{\rangle}

\def\dbar{\bar\partial}
\def\C{{\mathbb C}}

\def\Ok{{\mathcal O}}
\DeclareMathOperator{\rank}{rank}
\DeclareMathOperator{\ann}{ann}
\DeclareMathOperator{\supp}{supp}
\DeclareMathOperator{\Id}{Id}
\DeclareMathOperator{\Hom}{Hom}

\def\be{\begin{equation}}
\def\ee{\end{equation}}

\newtheorem{thm}{Theorem}[section]
\newtheorem{lma}[thm]{Lemma}

\newtheorem{prop}[thm]{Proposition}

\theoremstyle{definition}

\theoremstyle{remark}

\newtheorem{preremark}[thm]{Remark}
\newtheorem{preex}[thm]{Example}

\newenvironment{remark}{\begin{preremark}}{\end{preremark}}
\newenvironment{ex}{\begin{preex}}{\end{preex}}

\numberwithin{equation}{section}

\begin{document}

\title[Elementary construction of residue currents]{Elementary construction of residue currents associated to Cohen-Macaulay ideals}

\date{\today}

\author{Richard L\"ark\"ang}
\address{Richard L\"ark\"ang, Department of Mathematics,
Chalmers University of Technology and the University of Gothenburg, 412 96 G\"oteborg, Sweden.}

\email{larkang@chalmers.se}

\author{Emmanuel Mazzilli}
\address{Emmanuel Mazzilli,
    Laboratoire Paul Painlev\'e U.M.R. CNRS 8524, U.F.R. de Math\'ematiques, cit\'e scientifique,
    Universit\'e Lille 1, F59 655 Villeneuve d’Ascq Cedex, France.}

\email{emmanuel.mazzilli@math.univ-lille1.fr}


\keywords{}

\thanks{The first author was supported by the Swedish Research Council.}

\begin{abstract}
    For a Cohen-Macaulay ideal of holomorphic functions, we construct by elementary means
    residue currents whose annihilator is precisely the given ideal. We give two proofs that
    the currents have the prescribed annihilator, one using the theory of linkage,
    and another using an explicit division formula involving these residue currents 
    to express the ideal membership.
\end{abstract}

\maketitle

\section{Introduction}

Let $\Ok := \Ok_{\C^n,0}$ be the ring of germs of holomorphic functions at $0 \in \C^n$.
If $f \in \Ok$, and $U$ is a $(0,0)$-current
such that $f U = 1$, then it follows easily by regularity for the $\dbar$-operator on $(0,0)$-currents
that
\begin{equation} \label{eq:div1}
    g \dbar U = 0 \text{ if and only if } g \in J(f),
\end{equation}
where $J(f)$ is the principal ideal generated by $f$. For a current $T$, we let $\ann T$
denote the annihilator of $T$, i.e., all holomorphic functions $g$ such that $g T = 0$. Thus, if $f U = 1$,
we get that
\begin{equation*}
    \ann \dbar U = J(f).
\end{equation*}
One natural choice of such a current $U$ is the so-called principal value current $[1/f]$,
which is defined as
\begin{equation*}
    \left \la \left[\frac{1}{f}\right] , \phi \right\ra := \lim_{\epsilon \to 0^+} \int \frac{\chi(|f|^2/\epsilon)}{f} \wedge \phi,
\end{equation*}
where $\phi$ is a test form and $\chi$ is the cut-off function which is the characteristic function of the interval $[1,\infty)$,
or a smooth regularization of this function.
The existence of this current was proven by Dolbeault, \cite{Dol}, and Herrera-Lieberman, \cite{HL}.
That this limit exists relies on Hironaka's theorem about resolution of singularities, and is thus
far from elementary. Anyhow, any such choice of a current $U$ gives rise to a description of a
principal ideal $J(f)$.
A construction of such a current by elementary means, which in general is different from the
principal value current was done by the second author in \cite{MazCR}.

Consider now a complete intersection ideal $J$ of codimension $p$, i.e., $J = J(f_1,\dots,f_p)$ can be generated by exactly $p$
holomorphic functions, $f_1,\dots,f_p$. Coleff and Herrera showed in \cite{CH} that one can give a reasonable meaning to
$\dbar [1/f_p] \wedge \dots \wedge \dbar [1/f_1]$ in a similar way as for the principal value current.
Again, for all the different ways of regularizing the current, the existence of the limit relies on Hironaka's theorem.
In \cite{LS} it is described various ways that this product can be defined through some regularization procedure.
It was proven independently by Passare, \cite{PMScand} and Dickenstein-Sessa, \cite{DS}, that this so-called
Coleff-Herrera product satisfies the duality principle,
\begin{equation} \label{eq:ch-duality}
    \ann \dbar \left[\frac{1}{f_1}\right] \wedge \dots \wedge \dbar \left[\frac{1}{f_p}\right]
    = J(f_1,\dots,f_p).
\end{equation}
The proof of Passare relied on constructing an explicit division formula involving the Coleff-Herrera product
in order to obtain the ideal membership, while the proof in \cite{DS} essentially reduced to solving a series
of $\dbar$-equations.

Especially in relation to extension problems of holomorphic functions, it has turned out to be useful
to consider other currents for describing complete intersection ideals similar to \eqref{eq:ch-duality}.
It turns out that, generalizing the case of principal ideals in the beginning, if $J = J(f_1,\dots,f_p)$ is a complete
intersection ideal of codimension $p$, and if $X_k$ are $(0,k-1)$-currents for $k=1,\dots,p$ such that
\begin{equation} \label{eq:ci-conditions}
    f_1 X_1 = 1 \text{, } f_j X_k = 0 \text{ for $1 \leq j < k \leq p$ and } f_k X_k = \dbar X_{k-1} \text{ for } 2 \leq k \leq p,
\end{equation}
then
\begin{equation} \label{eq:ci-duality}
    \ann \dbar X_p = J(f_1,\dots,f_p).
\end{equation}
In \cite{MazJMAA}, the second author gave an elementary construction of such currents for any complete intersection ideal,
using only the much more elementary Weierstrass preparation theorem, and not relying on Hironaka's theorem.

Consider now a more general ideal $J = J(f_1,\dots,f_m)$, which is not necessarily a complete intersection ideal.
In \cite{AW1}, Andersson and Wulcan constructed, given a free resolution $(E,\varphi)$ of $\Ok/J$, a
($\Hom(E_0,E)$-valued) current $R^E$ such that
\begin{equation*}
    \ann R^E = J,
\end{equation*}
and two proofs of this description of the annihilator were given, one essentially reducing ideal membership to solving a series of $\dbar$-equations,
and the second by constructing an explicit division formula.
If $J = J(f_1,\dots,f_p)$ is a complete intersection ideal, and one takes the Koszul complex of $f$ as
a free resolution of $\Ok/J$, then $R^E$ equals the Coleff-Herrera product of $f$.
In general, although the current $R^E$ is explicitly expressed in terms of the free resolution $(E,\varphi)$,
it is in general quite difficult to understand, and the proof of existence of this current again relies
on Hironaka's theorem.

In \cite{LarComp}, the first author described a way of relating the currents $R^E$ of Andersson and Wulcan,
related to different free resolutions, of possibly different ideals. We consider the particular case
when $J$ is a Cohen-Macaulay ideal of codimension $p$, i.e., $\Ok/J$ has a free resolution $(E,\varphi)$
of length $p$. We also assume that $\rank E_0 = 1$, which is always possible to choose. One can always
find a complete intersection ideal $I = J(f_1,\dots,f_p)$ of codimension $p$
contained in $J$, for example by taking $p$ generic linear combinations of a set of generators of $J$,
cf., for example \cite{LarComp}*{Example~2}.
If one lets $(K,\psi)$ be the Koszul complex of $f$, then it is quite elementary homological
algebra that one can construct a morphism of complexes $a : (K,\psi) \to (E,\varphi)$ which extends the
natural surjection $\pi : \Ok/I \to \Ok/J$, i.e., which is such that the following diagram is commutative:
\begin{equation} \label{eq:amorphism}
    \begin{gathered}
\xymatrix{
    0 \ar[r] & E_p \ar[r]^{\varphi_p}& E_{p-1} \ar[r] & \cdots \ar[r]^{\varphi_1} & E_0 \ar[r] &\Ok/J \ar[r] & 0 \\
  0 \ar[r]  & K_p \ar[r]^{\psi_p} \ar[u]^{a_p}&  K_{p-1}\ar[r] \ar[u]^{a_{p-1}} & \cdots \ar[r]^{\psi_1}& K_0 \ar[u]^{a_0}  \ar[r] & \Ok/I \ar[u]^{\pi}\ar[r] & 0,
}
  \end{gathered}
\end{equation}
cf., Proposition~\ref{prop:comparison-morphism} below.
By \cite{LarComp}*{Example~3}, the current $R^E$ can then be described as
\begin{equation} \label{eq:comparison}
    R^E = a_p(e) \dbar \left[\frac{1}{f_1}\right] \wedge \dots \wedge \dbar \left[\frac{1}{f_p}\right],
\end{equation}
where $e_1,\dots,e_p$ is a frame for $K_1$ such that $\psi_1 = f_1 e_1^* + \dots + f_p e_p^*$, and
$e := e_p \wedge \dots \wedge e_1$ is the induced frame for $K_p \cong \bigwedge^p K_1$.
Hence, the current $R^E$ can be described as an explicit tuple of holomorphic functions times
a Coleff-Herrera product.
If $J = J(g_1,\dots,g_p)$ is also a complete intersection ideal of codimension $p$, and $(E,\varphi)$
is the Koszul complex of $g$, then $R^E$ is also a Coleff-Herrera product, and \eqref{eq:comparison}
then becomes the transformation law for Coleff-Herrera products, see \cite{LarComp}*{Remark~2}.

Our main result is the following combination of \eqref{eq:ci-duality} and \eqref{eq:comparison},
which thus with the help of the construction from \cite{MazJMAA} allows for constructing currents
representing Cohen-Macaulay ideals by elementary means, in particular not relying on Hironaka's 
theorem about resolution of singularities.

\begin{thm} \label{thm:main}
    Let $J$ be a Cohen-Macaulay ideal of codimension $p$, $(E,\varphi)$ be a free resolution
    of $\Ok/J$ such that $\rank E_0 = 1$, $I = J(f_1,\dots,f_p)$ a complete intersection ideal of codimension $p$
    contained in $I$, $(K,\psi)$ the Koszul complex of $f$, and let $a : (K,\psi) \to (E,\varphi)$
    be a morphism of complexes extending the natural surjection $\pi : \Ok/I \to \Ok/J$ as in \eqref{eq:amorphism}.
    If $X_1,\dots,X_p$ are currents satisfying \eqref{eq:ci-conditions}, then
    \begin{equation*}
        \ann a_p(e)\dbar X_p = J.
    \end{equation*}
\end{thm}

The requirement that $\rank E_0 = 1$ implies that the entries of $\varphi_1$ generate $J$,
and one can always find a free resolution such that this is the case.

We give two different proofs of this result, one in Section~\ref{sect:linkage},
which with the help of the theory of linkage reduces the problem to the complete
intersection case and \eqref{eq:ci-duality}, and as well a more direct proof
in Section~\ref{sect:integral} by means of an explicit division formula for
expressing the ideal membership.

In \cite{Lund}, Lundqvist defined by elementary means cohomological residues for a Cohen-Macaulay $J$,
which act on test forms which are $\dbar$-closed in a neighborhood of $\supp J$.
By the construction in \cite{Lund}, it follows easily that the action of the current $R^E$ on such 
test forms equals the residues by Lundqvist. These residues and its relation to other residues
is elaborated a bit in \cite{LExpl}*{Section~7}.
Since these cohomological residues are only defined acting on a restricted class of test forms,
the construction can be done by elementary means, depending only on finding a free resolution,
and in particular avoiding resolutions of singularities. This is at the cost of not showing
that these residues can act on arbitrary test forms. However, the main result in \cite{Lund} is that
even by only acting on this restricted class of test forms, one still obtains a duality theorem.

\section{Proof by the theory of linkage} \label{sect:linkage}

In this section, we give the first proof of Theorem~\ref{thm:main}, which is based on
the theory of linkage. In a somewhat different setting, similar methods were used
in \cite{LExpl}.
We recall that if $I$ and $J$ are two ideals in a ring $R$, then $I : J$ is the ideal
$I : J = \{ r \in R \mid rJ \subseteq I \}$.
The key result in proving Theorem~\ref{thm:main} is the following result, which
can be found in (the proof of) \cite{Vasc}*{Proposition~3.41}.

\begin{thm} \label{thm:colon}
    Let $J \subseteq \Ok$ be an ideal of pure codimension $p$,
    and let $I = J(f_1,\dots,f_p)$ be a complete intersection ideal of codimension $p$
    contained in $J$.
    If $K := I : J$, then $J = I : K$.
\end{thm}

We can describe the ideal $K$ appearing in Theorem~\ref{thm:colon} in a different way,
when $\Ok/J$ is Cohen-Macaulay. In order to do this, we use the following standard fact
from homological algebra, see \cite{Eis}, Proposition~A3.13.

\begin{prop} \label{prop:comparison-morphism}
    Let $\alpha : F \to G$ be a homomorphism of $\Ok$-modules, and let $(K,\psi)$ and
    $(E,\varphi)$ be free resolutions of $F$ and $G$.
    Then, there exists a morphism $a : (K,\psi) \to (E,\varphi)$ of complexes which extends $\alpha$.
\end{prop}

We will apply this in the case when $F = \Ok/I$, $G = \Ok/J$, $I \subseteq J$
and $\alpha$ is the natural surjection $\pi : \Ok/I \to \Ok/J$, as in \eqref{eq:amorphism}.
We remind for the following lemma, that for any ideal $J \subseteq \Ok$ of codimension $p$,
there always exists a complete intersection ideal $I \subseteq J$ of codimension $p$.
The following follows from Lemma~3.2 in \cite{FH}.

\begin{lma} \label{lma:colon}
    Let $J \subseteq \Ok$ be a Cohen-Macaulay ideal of codimension $p$,
    and assume that $I = J(f_1,\dots,f_p) \subseteq J$ is a complete intersection ideal of codimension $p$.
    Let $(E,\varphi)$ be a free resolution of $\Ok/J$ such that $\rank E_0 = 1$, and let $(K,\psi)$ be the Koszul complex of $f$,
    which is a free resolution of $\Ok/I$. Let $a : (K,\psi) \to (E,\varphi)$ be the morphism induced
    by the natural surjection $\Ok/I \to \Ok/J$ as in Proposition~\ref{prop:comparison-morphism}.
    Let $L$ be the ideal generated by the entries of $a_p$.
    Then,
    \begin{equation*}
        I : J = I + L.
    \end{equation*}
\end{lma}

\begin{remark}
    By reformulating this result, one can in fact drop the Cohen-Macaulay assumption, see
    \cite{LExpl}*{Lemma~4.6}, but for simplicity, we stick to this case here.
\end{remark}

    In \cite{DGSY}, a topic is treated which is related to this article,
    namely, given an analytic functional annihilated by some Cohen-Macaulay ideal, to express
    this functional in terms of residue currents, or more precisely Coleff-Herrera products.
    In order to do this, Lemma~\ref{lma:colon} plays an important role, see the proof of \cite{DGSY}*{Theorem~4.1}.
    In this article, we construct currents with a prescribed Cohen-Macaulay ideal $J$ as its
    annihilator. 
    From the currents we construct, one could construct an analytic functional annihilated by $J$
    and by our construction, this functional could be directly expressed in terms of
    a current $\dbar X_p$ whose annihilator is some complete intersection ideal contained in $J$. 
    The current $\dbar X_p$ could either be the current constructed in \cite{MazJMAA}
    or a Coleff-Herrera product, and in this latter case, one would obtain an expression 
    for the functional like in \cite{DGSY}.

\begin{proof}[Proof of Theorem~\ref{thm:main}]
    Since $J = I : (I:J)$ by Theorem~\ref{thm:colon}, and $I : J = I + L$ by
    Lemma~\ref{lma:colon}, $J = I : L$. We thus get that
    $g \in J$ if and only if all the entries of $g a_p(e)$ are in $I$.
    By \eqref{eq:ci-duality}, this holds if and only if $g a_p(e) \dbar X_p = 0$.
\end{proof}

\begin{ex}
    We consider now the most basic case, namely when $I = J(f_1,\dots,f_p)$ and $J = J(g_1,\dots,g_p)$
    are both complete intersection ideals of codimension $p$. Then $I \subseteq J$ is equivalent
    to that that $f = gA$ for some holomorphic $p \times p$-matrix $A$.
    In this case, when $(E,\varphi)$ and $(K,\psi)$ are the Koszul complexes of $g$ and $f$
    respectively, being free resolutions of $\Ok/J$ and $\Ok/I$ respectively, then the morphism
    $a : (E,\varphi) \to (K,\psi)$ extending the natural surjection $\pi : \Ok/I \to \Ok/J$
    is given by $a_k : \bigwedge^k \Ok^{\oplus p} \to \bigwedge^k \Ok^{\oplus p}$, $a_k = \bigwedge^k A$.
    In particular, $a_p = \det A$.
    Thus, reasoning as above, we get that
    \begin{equation} \label{eq:ci-trans-duality}
        g \in J \text{ if and only if }  (\det A) g \in I.
    \end{equation}
    This was an important part of the construction in \cite{MazJMAA},
    since \eqref{eq:ci-trans-duality} allowed to reduce the problem to constructing
    such currents for just for certain special ``adapted'' complete intersections.
\end{ex}

\begin{ex} \label{ex:cmcurve}
    Let $\pi : \C \to \C^3$, $\pi(t) = (t^3,t^4,t^5)$, and let $Z$ be the germ at $0$ of $\pi(\C)$.
    One can show that the ideal of holomorphic functions vanishing at $Z$
    equals $J = J(y^2-xz,x^3-yz,x^2y-z^2)$.
    The module $\Ok/J$ has a minimal free resolution $(E,\varphi)$ of the form
    \begin{equation*}
        0 \to \Ok^{\oplus 2} \xrightarrow[]{\varphi_2} \Ok^{\oplus 3} \xrightarrow[]{\varphi_1} \Ok \to \Ok/J,
    \end{equation*}
    where
    \begin{equation*}
        \varphi_2 = \left[ \begin{array}{cc} -z & -x^2 \\ -y & -z \\ x & y \end{array} \right]
            \text{ and }
            \varphi_1 = \left[ \begin{array}{ccc} y^2-xz & x^3-yz & x^2y-z^2 \end{array} \right].
    \end{equation*}

    In particular, since $\Ok/J$ has a minimal free resolution of length $2$,
    with $\rank E_2 = 2$, $\Ok/J$ is Cohen-Macaulay but $J$ is not a complete intersection.
    However, $Z$ is in fact a set-theoretic complete intersection, which one can
    see by verifying that indeed, if $f = (z^2-x^2y,x^4+y^3-2xyz)$, and $I = J(f)$,
    then $Z(I) = Z$.

    Let $(F,\psi)$ be the Koszul complex of $f$, which is a free resolution of $\Ok/I$ since $f$ is a complete intersection.
    One verifies that $a : (F,\psi) \to (E,\varphi)$ given by,
    \begin{equation*}
        a_2 = \left[ \begin{array}{c} x^3-yz  \\ y^2-xz \end{array} \right] \text{, }
        a_1 = \left[ \begin{array}{cc} 0 & y \\ 0 & x \\  -1 & 0 \end{array} \right] \text{ and }
        a_0 = \left[ \begin{array}{c} 1 \end{array} \right],
    \end{equation*}
    is a morphism of complexes extending the natural surjection $\pi : \Ok/I \to \Ok/J$.
    In the appendix of \cite{LarComp}, we give an example of how such a morphism can be computed
    with the help of the computer algebra system Macaulay2.

    By Theorem~\ref{thm:main}, we then get that
    \begin{align*}
        g \in J \text{ if and only if }
        (x^3-yz) g \in I \text{ and } (y^2-xz) g \in I.
    \end{align*}
\end{ex}

    For general Cohen-Macaulay ideals, one cannot expect that $Z(I) = Z(J)$ as in this example,
    since it by definition only is possible for set-theoretic complete intersections.

\subsection{Construction of the currents from \cite{MazCR} and \cite{MazJMAA}}

    In order to calculate the currents satisfying \eqref{eq:ci-conditions} as constructed in
    \cite{MazCR} and \cite{MazJMAA} for the complete intersection in the example above,
    we will first recall briefly the construction in general for a complete intersection ideal
    of codimension $2$. (The case of codimension $> 2$ is similar, but a bit more technically involved.)
    
    We thus consider a tuple $(f_1,f_2)$ of germs of holomorphic functions in $\C^n$ defining
    a complete intersection of codimension $2$.
    By a linear change of coordinates, we can assume we have coordinates $z$ on $\C^n$ such that
    $f_1$ and $f_2$ are of the
    form $f_i = v_i Q_i$, where $v_i$ are invertible, and $Q_i$ are Weierstrass polynomials in 
    $z_1$ for $i=1,2$. Then, the resultant $r_2$ of $Q_1$ and $Q_2$ (where $Q_1$ and $Q_2$ are considered as polynomials in $z_1$
    for the calculation of the resultant) is a holomorphic function independent
    of $z_1$. By a linear change of variables only in $(z_2,\dots,z_n)$, we can assume that $r_2(z) = u_2(z) P_2(z_2,\dots,z_n)$,
    where $P_2$ is a Weierstrass polynomial in $z_2$ independent of $z_1$ and $u_2$ is invertible.

    If we let $g_1 := f_1$ and $g_2 := r_2$, then $(g_1,g_2)$ is a complete intersection which
    satisfies that one can write $g_i(z) = u_i(z) P_i(z)$, where $u_i(z)$ is a unit,
    and $P_i(z)$ is a Weierstrass polynomials in $z_i$ of degree $N_i$, and in addition, $P_2(z)$ is independent of $z_1$.
    The construction of the currents $X_1$ and $X_2$ satisfying \eqref{eq:ci-conditions} is based on first constructing currents $Y_1,Y_2$ satisfying 
    the corresponding conditions for $(g_1,g_2)$, i.e.,
    $$ g_1 Y_1 = 1, g_1 Y_2 = 0 \text{ and } g_2 Y_2 = \dbar Y_1.$$
    To do this, one defines $Y_1$ by
    \begin{equation} \label{eq:Y1}
        \la Y_1,\phi \wedge dz_I \wedge d\bar{z}_J \ra := C_1 \int \frac{\overline{P}_1^{\gamma}}{g_1}
         \partial_{\bar{z}_1}^{N_1 \gamma}(\phi) dz_I \wedge d\bar{z}_J,
    \end{equation}
    and $Y_2$ by
    \begin{equation} \label{eq:Y2}
        \la Y_2,\phi \wedge dz_I \wedge d\bar{z}_J \ra := C_2 \int \frac{\overline{P}_2}{g_2} \partial_{\overline{z}_2}^{N_2}
        \left( \frac{\overline{\partial P_1^\gamma}}{g_1} \partial_{\overline{z}_1}^{N_1 \gamma}(\phi) \right) dz_I \wedge dz_J,
    \end{equation}
    where $\gamma$ is an integer chosen so the integrand in the definition of $Y_2$ becomes integrable, which indeed holds for $\gamma$ larger than $N_2$.
    The constant $C_1$ is chosen so that $f_1 X_1 = 1$ and the constant $C_2$ is then chosen such that $g_2 Y_2 = \dbar Y_1$.
    Through integration by parts one can calculate that $C_1 = (-1)^{N_1\gamma} (N_1\gamma)!$ and
    $C_2 = -(-1)^{N_2} C_1/(N_2!)$.

    In order to construct the currents $X_1,X_2$ for $(f_1,f_2)$, one then
    first writes $r_2 = a f_1 + b f_2$ for some holomorphic functions $a,b$,
    which indeed is possible, since by construction, $r_2 = a' Q_1 + b' Q_2$.
    Then, one defines $X_1 := Y_1$ and $X_2 := b Y_2$, which one can verify
    satisfies the properties \eqref{eq:ci-conditions}.

    Note that the greatest common divisor, i.e., the last remainder term in the Euclidean algorithm
    for $Q_1$ and $Q_2$ considered as polynomials in $z_1$ gives the resultant $r_2$ up to a constant
    which depends on the degrees of the polynomials appearing when running the algorithm, 
    see for example \cite{CLO}*{Exercise 3.6.10-11}.
    For our purposes it does not matter if we take a constant multiple of $r_2$ as $g_2$,
    and we will thus below take the greatest common divisor instead of the resultant.
    This is advantageous since the Euclidean algorithm is convenient for computation,
    and additionally, by going backwards in the algorithm,
    with the help of the terms that appear, one obtains a decomposition $r_2 = a' Q_1 + b' Q_2$.

\begin{ex} \label{ex:currents-cmcurve}
    We now consider the construction as described above for $f= (f_1,f_2) := (z^2-x^2y,x^4-2xyz+y^3)$, as in Expample~\ref{ex:cmcurve}.
    In order to make $f_1$ and $f_2$ Weierstrass polynomials in $x$ around zero (times units),
    we do the change of coordinates:
    $$x=X,\ \ y=Y,\ \ z=X+Z.$$
    In these new coordinates, $(f_1,f_2)$ becomes:
    $$f = ((1-Y)X^2+2XZ+Z^2,X^4-2X^2Y-2XYZ+Y^3 ).$$
    If we let $g=1-Y$, the Weierstrass polynomial in $X$ associated to $f_1$ is equal to $P_1:=X^2+{2XZ\over g}+{Z^2\over g}$, which
    has degree $N_1 = 2$.

    To calculate the resultant $r_2$ of $P_1$ and $f_2$, we use the Euclidean algorithm as mentioned above,
    and after an elementary but tedious calculations, we obtain that
    $$r_2 ={Z^2F^2\over g}-{2ZFG\over g}+G^2,$$
    where $F$ and $G$ are Weierstrass polynomials with respect to $Z$ defined by:
    $$F(Y,Z)={4\over g^2}(1-{2\over g})Z^3+2Y({2\over g}-1)Z,$$
    $$G(Y,Z)={1\over g^2}(1-{4\over g})Z^4+{2Y\over g}Z^2+Y^3.$$
    Since $g(0)=1$, it is easy to see that $r_2 =UP_2$ where $U$ is a unit near zero and $P_2$
    is a Weierstrass polynomial in $Z$ of degree $N_2 = 8$. 
    We finally do the linear coordinate change $(z_1,z_2,z_3) = (X,Z,Y)$ keeping the first variable fixed
    so that $P_1$ is a Weierstrass polynomial in $z_1$ and $P_2$ is a Weierstrass polynomial in
    $z_2$, and is independent of $z_1$.

    We thus let $(g_1,g_2) = (f_1,r_2)$ and define $Y_1,Y_2$ by \eqref{eq:Y1} and \eqref{eq:Y2}.
    By going backwards in the Euclidean algorithm, one finds that $r_2=af_1+bf_2$, where 
    $b=-(x+{2z\over g})F+G$. Thus, we get that
    \begin{equation*}
        X_1 := Y_1 \text{ and } X_2 := b Y_2
    \end{equation*}
    satisfies all the conditions \eqref{eq:ci-conditions}.
\end{ex}

\section{Proof by explicit division formulas} \label{sect:integral}

In this section, we give an explicit division formula which proves
Theorem~\ref{thm:main}.
The proof relies on the following two lemmas.
On $\C^n$ with coordinates $\zeta$, and for $z \in \C^n$ fixed,
$\delta_\eta$ denotes contraction with the vector field
$(\zeta_1 - z_1)\frac{\partial}{\partial \zeta_1} + \dots + (\zeta_n - z_n)\frac{\partial}{\partial \zeta_n}$.

\begin{lma} \label{lma:detaQH}
    Let $Q$ be a $(1,0)$-form on $\C^n$, and $H$ a holomorphic $(k+1)$-form.
    Then
    \begin{equation*}
        (n-k) \dbar (\delta_\eta Q) \wedge (\dbar Q)^{n-k-1} \wedge H
        = (\dbar Q)^{n-k} \wedge \delta_\eta H.
    \end{equation*}
\end{lma}

\begin{lma} \label{lma:nablaY}
    Let $I = J(f_1,\dots,f_p)$, $J$, $(E,\varphi)$, $(K,\psi)$ and $a : (K,\psi) \to (E,\varphi)$ be
    as in Theorem~\ref{thm:main}, and let $X_k$ be $(0,k-1)$-currents for $k=1,\dots,p$, which satisfy \eqref{eq:ci-conditions}.
    Let $Y_k$ be defined as
\begin{equation*}
    Y_k := a_k ( e_1 \wedge \dots \wedge e_k \wedge X_k ).
\end{equation*}
    Then $Y$ satisfies
\begin{equation} \label{eq:nablaY}
    \nabla Y := 1 - \dbar Y_p,
\end{equation}
where $Y = Y_1 + \dots + Y_p$ and $\nabla = \varphi-\dbar$.
\end{lma}

More explicitly, the equation \eqref{eq:nablaY} means that
\begin{equation} \label{eq:nablaconditions}
    \varphi_1 Y_1 = 1 \text{ and } \varphi_k Y_k =
    \dbar Y_{k-1} \text{ for $2\leq k \leq p$ }.
\end{equation}
We let $e_1,\dots,e_p$ be the standard basis of $K_1 \cong \Ok^p$ such that
the morphism in $(K,\psi)$ is contraction with $\sum f_i e_i^*$,
and in particular, $K_k$ has as a basis $e_{I_1} \wedge \dots \wedge e_{I_k}$
for $1 \leq I_1 < \dots < I_k \leq p$.

\begin{remark} \label{rem:identification}
To be precise, $\varphi_1 Y_1$ is a $E_0$-valued $(0,0)$-current.
However, since we assume that $\rank E_0 = 1$, we have that
$E_0 \cong \Ok \cong K_0$. Note that $K_0 \cong \bigwedge^0 K_1$ has
a canonical frame, $e_\emptyset$. In addition, $a_0 : K_0 \to E_0$ is an isomorphism,
so $a_0$ induces a frame $a_0(e_\emptyset)$ of $E_0$. In order
to simplify the notation, we have identified $\Ok \stackrel{\cong}{\to} E_0$ through
the map, $f \mapsto f a_0(e_\emptyset)$, so that we write $\varphi_1 Y_1 = 1$
instead of $\varphi_1 Y_1 = a_0(e_\emptyset)$.
\end{remark}

\medskip

In order to prove the division formula, we will also use the so-called \emph{generalized Hefer forms}
associated to a free resolution $(E,\varphi)$ as introduced by Andersson in \cite{AndIntII}. They consist of
$(k-\ell,0)$-form valued holomorphic morphisms $H^\ell_k : E_k \to E_\ell$, satisfying $H^\ell_k = 0$
if $k < \ell$, $H^\ell_\ell = I_{E_\ell}$ and
\begin{equation} \label{eq:Hefer}
    \delta_\eta H^\ell_{k+1} = H^\ell_k \varphi_{k+1}(\zeta) - \varphi_{\ell+1}(z) H^{\ell+1}_{k+1}
\end{equation}
for $k > \ell$.

In a similar way to in \cite{MazCR} and \cite{MazJMAA}, we will show that
\eqref{eq:nablaY} and the following integral representation formula
lead to our sought after division formulas. Although similar formulas
exist also when $D$ is strongly pseudoconvex, \cites{BA,DH}, for simplicity
of the presentation, we stick to the case when $D$ is convex,
see for example \cite{BWeighted}*{Chapter~4}.

\begin{thm} \label{thm:w-div-form}
    Let $D \subseteq \C^n$ be a smooth convex domain with defining function $\rho$,
    and let $Q := \partial \log (1/(-\rho))$ and
    \begin{equation*}
        P^{N,\ell}(\zeta,z) := \frac{1}{(\delta_\eta Q + 1)^{N+\ell}} (\dbar Q)^{\ell},
    \end{equation*}
    which is holomorphic in $z \in D$, and for $N \gg 1$, it is smooth in $\zeta \in \overline{D}$
    and vanishes to arbitrarily high order (depending on $N$) on $\partial D$.
    If $h \in \Ok(\overline{D})$, then
    \begin{equation} \label{eq:w-div-form}
        h(z) = c_{N,n} \int_D P^{N,n}(\zeta,z) h(\zeta) \text{, \quad for $z \in D$,}
    \end{equation}
    where $c_{N,n} := { N+n-1 \choose n}$.
\end{thm}

We now give the second proof of Theorem~\ref{thm:main}. More precisely, we have the
following result, which implies Theorem~\ref{thm:main}.

\begin{thm} \label{thm:div-form}
    Let $D \subseteq \C^n$ be a smooth convex domain, $(E,\varphi)$ a free resolution of $\Ok/J$ for some ideal
    $J$, and assume that $\rank E_0 = 1$. Let $Y_1,\dots,Y_p$ be
    currents satisfying \eqref{eq:nablaY}, and take $P^{N,n}$ as in Theorem~\ref{thm:w-div-form}, where
    $N \gg 1$ is such that $P^{N,n}$ vanishes on $\partial D$ to order higher
    than the order of $Y_1,\dots,Y_p$ on $\overline{D}$, and let $H$ be a generalized
    Hefer form for $(E,\varphi)$.
    If $h \in \Ok(\overline{D})$, then
    \begin{equation} \label{eq:div-form}
        h(z) = \varphi_1(z) P(h)(z) + R(h)(z) \text{, \quad for $z \in D$,}
    \end{equation}
    where $R(h)(z) =: R(z)$ is a holomorphic function given by
    \begin{equation} \label{eq:Rp}
        R(z) = c_p\int_D h(\zeta) P^{N,n-p}(\zeta,z) H^0_p \dbar Y_p,
    \end{equation}
    and $P(h)(z) =: P(z)$ is given as
    \begin{equation*}
        P(z) = P_0(z) + \dots + P_{p-1}(z),
    \end{equation*}
    where $P_k(z)$ is a vector of holomorphic functions given by
\begin{equation} \label{eq:Pk}
    P_k(z) = c_k\int_D h(\zeta) P^{N,n-k}(\zeta,z) H^1_{k+1} Y_{k+1},
\end{equation}
    for suitably chosen constants $c_0,\dots,c_p$.
\end{thm}

We recall that if $(E,\varphi)$ is a free resolution of an ideal $\Ok/J$ and $\rank E_0 = 1$,
then the entries of $\varphi_1$ are generators of $J$, so the first term in the right-hand side
of \eqref{eq:div-form} belongs to $J$.

\begin{proof}[Proof of Theorem~\ref{thm:main}]
Since the kernel defining $R(h)(z)$ is smooth in $\zeta$ except for the term $Y_p$,
and if we thus as in Lemma~\ref{lma:nablaY} take $Y_p = a_p(e) X_p$, we get by the
division formula that $\ann a_p(e) \dbar X_p \subseteq J$. Conversely, by the inclusion
$J \subseteq I : (I : J)$, and Lemma~\ref{lma:colon}, if $h \in J$, then $a_p(e) h \in I$,
so $a_p(e) h \in \ann \dbar X_p$ by \eqref{eq:ci-duality}.
\end{proof}

\begin{remark}
It might seem like we also for this proof use the theory of linkage,
using Lemma~\ref{lma:colon}, and the inclusion $J \subseteq I : (I: J)$.
However, Lemma~\ref{lma:colon} is rather straight-forward homological algebra,
and the inclusion $J \subseteq I : (I : J)$ is trivial, while the real use of the
theory of linkage in the previous section was the non-trivial inclusion
$I : (I : J) \subseteq J$. In this section, we prove Theorem~\ref{thm:main} with
the help of integral formulas, instead of using this inclusion. We then note
that Theorem~\ref{thm:main} indeed implies this inclusion.
By \eqref{eq:ci-duality} and Lemma~\ref{lma:colon}, we get that
$h a_p(e) \dbar X_p = 0$ is equivalent to that $h \in I : (I : J)$.
By Theorem~\ref{thm:main}, we thus get that $h \in J$, proving the desired inclusion.
\end{remark}

\begin{proof}[Proof of Theorem~\ref{thm:div-form}]
We define for $0 \leq k \leq p-1$,
\begin{equation} \label{eq:Rkdef}
    R_k(z) := c_k\int h(\zeta) P^{N,n-k}(\zeta,z) H^0_k
    \varphi_{k+1}(\zeta) Y_{k+1},
\end{equation}
where the constants $c_k$ are the same as the constants $c_k$ in \eqref{eq:Pk},
and these constants will be determined below.
For $k=p$, we let $R_p(z) := R(z)$, where
$R(z)$ is given by \eqref{eq:Rp}.

We start by defining $c_0 := C_{N,n}$, where $C_{N,n}$ is the constant in \eqref{eq:w-div-form}.
Since $H^0_0 = \Id_{E_0}$, and $\varphi_1 Y_1 = 1$, 
we then get by \eqref{eq:w-div-form} that 
\begin{equation} \label{eq:repr0}
    h(z) = R_0(z) = c_0\int h(\zeta) P^{N,n}(\zeta,z) \varphi_1(\zeta) Y_1.
\end{equation}

Having chosen $c_0$, the proof then proceeds by induction, by proving that for $0 \leq k < p$
we can choose $c_{k+1}$ such that
\begin{equation} \label{eq:induction}
    R_k(z) = \varphi_1(z) P_k(z) + R_{k+1}(z),
\end{equation}
where $P_k$ is given by \eqref{eq:Pk}. To see this, we note first that by using \eqref{eq:Hefer} on the term $H^0_k \varphi_{k+1}(\zeta)$
in \eqref{eq:Rkdef}, we get that
\begin{equation} \label{eq:induction2}
    R_k(z) = \varphi_1(z) P_k(z) + c_k \int h(\zeta) P^{N,n-k}(\zeta,z) \delta_\eta H^0_{k+1} Y_{k+1}.
\end{equation}
By Lemma~\ref{lma:detaQH}, we then obtain since $H^0_{k+1}$ is a row of holomorphic $(k+1)$-forms that
\begin{equation*}
    \begin{gathered}
        \dbar P^{N,n-k-1}(\zeta,z) \wedge H^0_{k+1} =
         -\frac{(N+n-k-1)}{(\delta_\eta Q + 1)^{N+n-k}} (\dbar\delta_\eta Q) \wedge (\dbar Q)^{n-k-1} \wedge H^0_{k+1} = \\
        = \frac{C_k}{(\delta_\eta Q + 1)^{N+n-k}}
        (\dbar Q)^{n-k} \wedge \delta_\eta H^0_{k+1}
        = C_k P^{N,n-k}(\zeta,z) \wedge \delta_\eta H^0_{k+1},
    \end{gathered}
\end{equation*}
where $C_k = \frac{-(N+n-k-1)}{(n-k)}$.
Inserting this in \eqref{eq:induction2}, we get that
\begin{equation} \label{eq:induction3}
    R_k(z) = \varphi_1(z) P_k(z) + (c_k/C_k) \int h(\zeta) \dbar P^{N,n-k-1}(\zeta,z) \wedge H^0_{k+1} Y_{k+1}.
\end{equation}

If $\Psi$ is a smooth $(k+1,0)$ form on $\overline{D}$, then by extending
$\Psi \wedge P^{N,n-k-1}$ by $0$ outside of $\overline{D}$, by the choice of $N$, this extension
is a form which is differentiable to a higher order than the order of $Y_{k+1}$ for $0 \leq k < p$.
Thus, we can consider the extension of $\Psi \wedge P^{N,n-k-1}$ as a test form of bidegree $(n,n-k-1)$,
and we thus get by definition of $\dbar Y_{k+1}$ that
\begin{equation} \label{eq:intbyparts}
    \int_{\overline{D}} \Psi \wedge P^{N,n-k-1} \wedge \dbar Y_{k+1} = \pm \int_{\overline{D}} \dbar(\Psi \wedge P^{N,n-k-1}) \wedge Y_{k+1}.
\end{equation}

Since $h$ and $H^0_{k+1}$ are holomorphic, we get by applying \eqref{eq:intbyparts} to the last term in \eqref{eq:induction3} that 
\begin{equation} \label{eq:induction4}
    R_k(z) = \varphi_1(z) P_k(z) \pm (c_k/C_k) \int h(\zeta) P^{N,n-k-1}(\zeta,z) \wedge H^0_{k+1} \dbar Y_{k+1},
\end{equation}
for $0 \leq k < p$.
If we then let $c_{k+1} := \pm c_k/C_k$, and for $0 \leq k \leq p-2$ use that
$\dbar Y_{k+1} = \varphi_{k+2}(\zeta) Y_{k+2}$ by \eqref{eq:nablaconditions}, we obtain that the right-most term
in \eqref{eq:induction4} equals $R_{k+1}(z)$. For $k = p-1$, we see directly from \eqref{eq:Rp} that 
the right-most term equals $R(z) = R_p(z)$. We have thus proven that \eqref{eq:induction} holds for
$k=0,\dots,p-1$.

To conclude, starting with \eqref{eq:repr0}, then using \eqref{eq:induction} repeatedly for $k=0,\dots,p-1$,
and finally that $R_p(z) = R(z)$, we obtain \eqref{eq:div-form}.
\end{proof}

We finally also remark that indeed, using the framework of integral formulas of Andersson,
as in \cite{AndIntII}, it follows from \eqref{eq:nablaY} that one has a division formula
\begin{equation*}
    h(z) = \varphi_1(z) \sum_k \int H^1_k Y_k(\zeta) h(\zeta) \wedge g_{n-k} + \int H^0_p \dbar Y_p(\zeta) h(\zeta) \wedge g_{n-p},
\end{equation*}
where $g$ is a weight with compact support as in \cite{AW1}*{Section~5}.
Here we have preferred to give a more direct proof based on the basic Theorem~\ref{thm:w-div-form},
avoiding the need to use this full machinery.

\begin{proof}[Proof of Lemma~\ref{lma:detaQH}]
    We note first that since $\delta_\eta$ is an anti-derivation,
    $\delta_\eta \dbar = -\dbar \delta_\eta$. In addition, for degree-reasons,
    $Q \wedge (\dbar Q)^{n-k-1} \wedge H = 0$. Thus,
    \begin{equation*}
        0 = \delta_\eta \dbar( Q \wedge (\dbar Q)^{n-k-1} \wedge H)
        = - \dbar \delta_\eta (Q \wedge (\dbar Q)^{n-k-1} \wedge H).
    \end{equation*}
    Hence, since $\delta_\eta$ is an anti-derivation, and $Q$ and $\dbar Q$
    are of odd and even degree respectively, and $(\dbar Q)^{n-k-1} \wedge H$
    is $\dbar$-closed,
    \begin{equation*}
        (\dbar \delta_\eta Q) \wedge (\dbar Q)^{n-k-1} \wedge H -
        \dbar ( Q \wedge \delta_\eta (\dbar Q)^{n-k-1} \wedge H) =
        (\dbar Q)^{n-k} \wedge \delta_\eta H.
    \end{equation*}
    It then only remains to see that
    \begin{equation} \label{eq:detadbarQ}
        -\dbar ( Q \wedge \delta_\eta (\dbar Q)^{n-k-1} \wedge H) =
        (n-k-1) \dbar \delta_\eta Q \wedge (\dbar Q)^{n-k-1} \wedge H.
    \end{equation}
    To see this, we first note that since $\dbar Q$ has even degree,
    \begin{equation*}
        \dbar ( Q \wedge \delta_\eta (\dbar Q)^{n-k-1} \wedge H) =
        (n-k-1) \dbar (Q \wedge (\delta_\eta \dbar Q) \wedge (\dbar Q)^{n-k-2} \wedge H)
    \end{equation*}
    In addition, $\delta_\eta \dbar Q = - \dbar \delta_\eta Q$, which is
    $\dbar$-closed, so
    \begin{equation*}
        - \dbar (Q \wedge (\delta_\eta \dbar Q) \wedge (\dbar Q)^{n-k-2} \wedge H)
        = \dbar Q \wedge \dbar \delta_\eta Q \wedge (\dbar Q)^{n-k-2} \wedge H,
    \end{equation*}
    which gives \eqref{eq:detadbarQ}.
\end{proof}

\begin{proof}[Proof of Lemma~\ref{lma:nablaY}]
    To prove that $Y$ satisfies \eqref{eq:nablaconditions},
    we note first that since $\varphi_1 a_1 = \psi_1$ (where we identify $E_0$ and $K_0$ with $\Ok$ as in Remark~\ref{rem:identification}),
    \begin{equation*}
        \varphi_1 Y_1 = \varphi_1 a_1( e_1 \wedge X_1 ) = \psi_1 e_1 X_1 =
        f_1 X_1 = 1.
    \end{equation*}
    For $2 \leq k \leq p$, we get that
    \begin{align*}
        \varphi_k Y_k &= a_{k-1}(\psi_k (e_1 \wedge \dots \wedge e_k)) \wedge X_k \\
        &= \sum_{j=1}^k (-1)^{j-1} f_j a_{k-1}(e_1 \wedge \dots \wedge
        \widehat{e}_j \wedge \dots \wedge e_k) \wedge X_k \\
        &= (-1)^{k-1} a_{k-1} (e_1 \wedge \dots \wedge e_{k-1}) \wedge f_k X_k \\
        &= (-1)^{k-1} a_{k-1} (e_1 \wedge \dots \wedge e_{k-1}) \wedge \dbar X_{k-1}
        = \dbar Y_{k-1},
    \end{align*}
    where all the other terms in the sum vanish since $f_j X_k = 0$ for $j < k$,
    and the sign in the last equality is due to the superstructure, cf.,
    for example \cite{LarComp}*{Section~2.1}, since $a_{k-1}(e_1 \wedge \dots \wedge e_{k-1})$
    has degree $k-1$.
\end{proof}

\end{document}